\documentclass[11pt]{article}
\usepackage[letterpaper,margin=1in]{geometry}

\usepackage{amsmath}
\usepackage{amssymb}
\usepackage{amsfonts}
\usepackage{amsthm}
\usepackage{algpseudocode,algorithm,algorithmicx}
\usepackage{bbm}
\usepackage{mathbbol}
\usepackage{xcolor}
\usepackage{wrapfig}
\usepackage{authblk}
\usepackage{hyperref}

\newtheorem{theorem}{Theorem}
\newtheorem{lemma}{Lemma}
\newtheorem{claim}{Claim}
\newtheorem{crl}{Corollary}
\newtheorem{rem}{Remark}

\DeclareMathOperator{\R}{\mathbb{R}}
\DeclareMathOperator{\Z}{\mathbb{Z}}
\DeclareMathOperator{\N}{\mathbb{N}}

\DeclareMathOperator {\me}{\mathbf{e}}
\DeclareMathOperator{\len}{len}
\newcommand{\0}{\mathbb{0}}

\newcommand{\st}{\ensuremath{\mbox{s.t.}}}
\renewcommand{\H}{\ensuremath{\mathcal{H}}}

\newcommand{\lpg}{(LP$_g$)}
\newcommand{\opt}{\ensuremath{\mbox{opt}}}
\newcommand{\gk}{\ensuremath{\mbox{GK}}}

\title{Fast Approximation Algorithms for
the Generalized Survivable Network Design Problem}

\begin{document}

\author[1]{Andreas Emil Feldmann}
\author[2]{Jochen K\"onemann}
\author[2]{Kanstantsin Pashkovich}
\author[2]{Laura Sanit\`a}

\affil[1]{SZTAKI, Hungarian Academy of Sciences, Hungary \& KAM, Charles 
University in Prague, Czechia. \texttt{feldmann.a.e@gmail.com}}
\affil[2]{Department of Combinatorics and Optimization, University of Waterloo, 
Canada. \texttt{\{jochen, kpashkovich, laura.sanita\}@uwaterloo.ca}}


\maketitle

\begin{abstract}
  In a standard {\em $f$-connectivity} network design problem, we are
  given an undirected graph $G=(V,E)$, a {\em cut-requirement
    function} $f:2^V \rightarrow \N$, and non-negative costs $c(e)$ for
  all $e \in E$. We are then asked to find a
  minimum-cost vector $x \in \N^E$ such that
  $x(\delta(S)) \geq f(S)$ for all~$S \subseteq V$. We
  focus on the class of such problems where $f$ is a {\em proper}
  function. This encodes many well-studied NP-hard problems such as
  the {\em generalized survivable network design} problem.

  In this paper we present the first {\em strongly polynomial time}
  FPTAS for solving the LP relaxation of the standard IP formulation
  of the $f$-connectivity problem with general proper functions~$f$. 
Implementing Jain's algorithm, this yields a strongly
  polynomial time $(2+\epsilon)$-approx\-imation for the generalized
  survivable network design problem (where we consider {\em rounding
    up} of rationals an arithmetic operation).
\end{abstract}

\section{Introduction}

The input in a typical {\em network design} problem consists of a
directed or undirected graph $G=(V,E)$, non-negative unit-capacity
installation costs $c(e)$ for all $e \in E$, and a collection of {\em
  connectivity requirements} among the vertices in $V$. The goal is
then to find a minimum-cost capacity installation in $G$ that
satisfies the connectivity requirements. The above abstract problem
class captures many practically relevant optimization problems, many
of which are NP-hard. Therefore, maybe not surprisingly, there has
been a tremendous amount of research in the area of approximation
algorithms for network design problems throughout the last four decades
(e.g., see~\cite{Gu11,Ho96}). 

\begin{wrapfigure}[6]{R}{0.45\textwidth}
\vspace{-9mm}
\begin{align}
  \min ~ & \sum_{e \in E} c(e)x(e) \tag{IP} \label{eq:nsdp_ip} \\
  \st ~ & x(\delta(S)) \geq f(S) ~~ \forall S \subset V
          \notag \\
  & x \geq \0, x \mbox{ integer.} \notag
\end{align}
\end{wrapfigure}

Connectivity requirements can
be modelled in many ways, but we will adopt the {\em $f$-connectivity}
viewpoint in this paper. Here,  one is given a {\em cut-requirement function}
$f : 2^V \rightarrow \N$, 
and one wants to find a minimum-cost non-negative integer vector $x$
such that for each $S \subseteq V$, the sum of variables $x(e)$ for edges $e$ 
crossing the cut $S$ is at least $f(S)$. In other words, we are
interested in problems that can be encoded by integer 
program \eqref{eq:nsdp_ip}; here $\delta(S)$ denotes the set of edges incident 
to a vertex in $S$ and a vertex outside $S$, and 
$x(\delta(S)):=\sum_{e\in\delta(S)} x(e)$.

Restricting even further, we will henceforth only be concerned with
instances of \eqref{eq:nsdp_ip} where $f$ is {\em proper}: 
$f$ is proper if it satisfies the three properties of {\em maximality} (i.e., 
$f(A \cup B) \leq \max \{f(A),f(B)\}$
for all disjoint sets $A,B \subseteq V$), {\em symmetry} 
(i.e., $f(S)=f(V\setminus S)$ for all $S \subseteq V$), and $f(V)=0$.
Program \eqref{eq:nsdp_ip} 
with proper cut-requirement functions $f$ captures
(among others) the special case 
where the goal is to find a minimum-cost
network that has $r(u,v)$ edge-disjoint paths connecting any pair
$u,v$ of vertices (for given non-negative integer parameters $r$). 
The implicit cut-requirement function in this case is then given by
$f(S):=\max_{u \in S, v \in V\setminus S} r(u,v)$ for all $S \subseteq V$.

Based on the {\em primal-dual} method, Goemans and
Williamson~\cite{GW95} first gave a $2\H(f_{\max})$-approxima\-tion
algorithm for \eqref{eq:nsdp_ip} with proper cut-requirement functions
where one is allowed to pick edges multiple times in the solution.
Goemans et al.~\cite{GSGT1994} later obtained the same performance ratio for the
setting where multiple copies of edges are not allowed. More recently,
in a breakthrough result, Jain~\cite{Jain} obtained a
$2$-approximation for the more general class of {\em
  skew-supermodular} cut-requirement functions based on {\em iterative 
rounding}.
\begin{wrapfigure}[7]{R}{0.45\textwidth}
\vspace{-5mm}
\begin{align}
  \min ~ & \sum_{e \in E} c(e) x(e)
           \tag{$\mbox{LP}_1$} \label{eq:jain}\\
  \st ~& x(\delta(S)) \geq f(S) - z(\delta(S)) ~~ \forall S \subseteq V \notag
  \\
  & x(e) = 0 ~~ \forall e \in I\notag \\
  & x \geq \0 \notag
\end{align}
\end{wrapfigure}
Jain's algorithm iteratively fixes the value of a subset of variables in
\eqref{eq:nsdp_ip}. To aid this, he first defines a slightly more
general version of the IP, where the value of certain variables is
fixed. Specifically, for a set $I \subseteq E$ of edges, assume that
the value of variable $x(e)$ is fixed to $z(e) \in \N$. The LP
relaxation of the IP for the corresponding residual problem is given
in~\eqref{eq:jain}. Jain's key observation was that the extreme points
of the feasible region of \eqref{eq:jain} are {\em sparse}, and have
at least one variable with value at least $1/2$.

Capitalizing on this insight, his algorithm then iteratively solves $O(|V|)$ 
instantiations of \eqref{eq:jain} while intermittently
rounding up the values of large variables in the computed solutions.
In order to solve \eqref{eq:jain} one needs to employ the Ellipsoid
method~\cite{GLS88} together with a polynomial-time seperation oracle
for the LP's constraints (see~\cite{Gabow}).

Our work is motivated by {\em Open Problem 4} in Williamson and
Shmoys' recent book~\cite{Williamson} where the authors point out that
solving \eqref{eq:jain} for general (proper) functions $f$ may be
computationally quite demanding despite the fact that it can be done
efficiently in a theoretical sense. The authors leave as an open
problem the design of a {\em primal-dual} $2$-approximation for the
survivable network design problem.  Our main result is a replacement
of the Ellipsoid-based exact LP-solver calls in Jain's algorithm by
approximate ones that are based on the (in a sense) primal-dual 
 {\em multiplicative-weights} method of~\cite{Garg}. 
We realize that the likely intended meaning of {\em primal-dual} in Williamson 
and
Shmoys' open problem statement is the {\em primal-dual method
for approximation algorithms} (as in~\cite{GW95}). However
we believe that the contribution made in this paper is in line with
the motivation given for Open Problem 4 in~\cite{Williamson}:
we substantially speed up LP computations in Jain's algorithm at the
expense of an inconsequential loss in performance guarantee of the 
algorithm. 

\begin{theorem}\label{thm:main}
  For any $\epsilon>0$, there 
  is a $(1+\epsilon)$-approximation algorithm 
  for \eqref{eq:jain} that runs in {\em strongly polynomial
  time}\footnote{An algorithm is \emph{strongly polynomial} if its
  number of arithmetic operations, i.e. the number of additions,
  subtractions, multiplications, divisions and comparisons, is bounded
  by a polynomial in the dimension of the problem (i.e., the number of
  data items in the input), and the length of the numbers occurring
  during the algorithm is bounded by a polynomial in the length of the
  input. } 
independently of the values of $c$ and $f$.
\end{theorem}

In contrast to our result, Jain~\cite{Jain} observes that the relaxations of 
\eqref{eq:nsdp_ip} and \eqref{eq:jain} are of
a {\em combinatorial} nature, and hence can be solved in strongly
polynomial-time via Tardos' algorithm~\cite{Ta86} {\em whenever} their
number of variables and constraints are polynomially bounded (in the
problem dimension). For example, \eqref{eq:nsdp_ip} and \eqref{eq:jain} have an
equivalent compact representation when
$f(S)=\max_{u \in S, v \not\in S} r(u,v)$ for all $S \subseteq V$. We also
note that one can argue that the Ellipsoid method applied to 
 \eqref{eq:jain} and the linear relaxation of \eqref{eq:nsdp_ip} terminates in a strongly polynomial number of
steps {\em whenever} function $f(S)$ is polynomially bounded (in the problem
dimension), for all $S \subseteq V$ as this implies small encoding-length of 
vertices
of the feasible region of \eqref{eq:nsdp_ip} and \eqref{eq:jain}.

To achieve the result in \autoref{thm:main}, we rely on the
\emph{multiplicative weights method} of Garg and
K\"onemann~\cite{Garg} (henceforth referred to as \gk). This is a natural idea 
as \eqref{eq:jain} belongs to the class of {\em
  positive covering LPs}. As such, \cite{Garg} applies to the LP dual
of \eqref{eq:jain}. The algorithm can therefore be used to compute an
approximate pair of primal and dual solutions in strongly-polynomial
time as long as we are able to provide it with a strongly-polynomial
time (approximation) algorithm for the so called \emph{shortest row
  problem}. For \eqref{eq:jain} this boils down to computing
\[
	\min_{\substack{f(S)-z(\delta(S))\ge 1\\S\subseteq V}} 
\frac{x(\delta(S))}{f(S)-z(\delta(S))},
\]
for given $x \in \R^E_+$ and $z\in\Z^E_+$, i.e., finding a corresponding set 
$S$. The above shortest row problem is solved
directly by Gabow et al.'s strongly-polynomial time separation oracle
for the constraints of \eqref{eq:jain} (see~\cite{Gabow}) in the case
where $I=\emptyset$, and hence $z=\0$.  Once $I \neq \emptyset$, Gabow
et al.'s algorithm can not be used directly to give a strongly
polynomial-time algorithm, and a more subtle approach is needed.  In
fact, in this case, we provide only a $(1+\zeta)$-approximate solution
to the shortest row problem (for appropriate $\zeta>0$). As is
well-known (e.g., see~\cite{Fleischer,Garg}), the exact shortest-row
subroutine used in \gk\ may be replaced by an $\alpha$-approximate
one, sacrificing a factor of $\alpha$ in the overall performance ratio
of the algorithm in~\cite{Garg}.  We obtain the following direct
corollary of \autoref{thm:main}.

\begin{crl}\label{crl:main}
  Combining \autoref{thm:main} with Jain's algorithm, we obtain a
  strongly polynomial-time\footnote{if rounding up numbers is
    considered an arithmetic operation.}
 $(2+\varepsilon)$-approximation
  algorithm for~\eqref{eq:nsdp_ip}, that does not use linear programming
  solvers.  
\end{crl}

We once again stress that the above results hold for {\em any not
  necessarily bounded} proper
cut-requirement function $f$. 

\paragraph*{Further related work}

The past 30 years have seen significant research on solving linear
programs efficiently; e.g., see the work by Shahrokhi \&
Matula~\cite{SM90}, Luby \& Nisan~\cite{LN93}, Grigoriadis \&
Khachian~\cite{GK94,GK96}, Young~\cite{Yo95,Yo01}, Garg \&
K\"onemann~\cite{Garg}, Fleischer~\cite{Fleischer}, and Iyengar \&
Bienstock~\cite{BI04}. We refer the reader to two recent surveys 
by Bienstock~\cite{Bi06} and Arora, Hazan \& Kale~\cite{AHK12}. 

Particularly relevant to this paper is the work by
Fleischer~\cite{Fl04} who previously proposed a Lagrangian-based
approximation scheme for positive covering LPs with added {\em
  variable upper bounds}. Their algorithm builds on~\cite{Garg} and
\cite{Fleischer}, and achieves a performance ratio of $(1+\epsilon)$
for any positive $\epsilon$ using $O(\epsilon^{-2}m\log (Cm))$ calls
to a separation oracle for the given covering problem; here $m$
denotes the number of variables, and $C$ is bounded by the maximum
objective function coefficient. Garg and Khandekar~\cite{GK04}
later addressed the same problem, and presented an improved algorithm
with $O(m\epsilon^{-2}\log m + \min\{n, \log\log C\})$ calls to an
oracle for the most violated constraint.

The algorithms in~\cite{Fl04,GK04} naturally apply to
solving LP relaxations of various network design IPs where the
multiplicity $x(e)$ of each edge $e$ is limited to some given upper
bound. As the approaches in~\cite{Fl04,GK04} need to approximate the
same type of the shortest row problem, as an immediate corollary of
our result, we obtain a strongly polynomial-time
\mbox{$(2+\varepsilon)$-approximation} algorithm for
\eqref{eq:nsdp_ip} with constant upper bounds on the variables. This
captures in particular the interesting case in which we have
\emph{binary} constraints for $x$.

Finally, we also mention the work of Garg \& Khandekar~\cite{GK02} who
present a fully polyonimal-time approximation algorithm for the
fractional {\em Steiner forest} problem. The algorithm also applies to
the more general problem of finding a minimum-cost fractional hitting
set of a given collection of {\em clutters}. 

\paragraph*{Organization}

We first provide some more details on how to implement the iterative
rounding algorithm of Jain.  
We continue and provide a detailed description of
\gk\ with approximate oracles in \autoref{sec:mult_weight} for
completeness, and describe the shortest-row oracles in
\autoref{sec:shortest_row}. Finally, in \autoref{sec:wrap-up} we put together 
all ingredients to prove our main result.

\section{Iterative rounding}\label{sec:rounding}

Recall that Jain's key structural insight was
to observe that extreme points $x \in
\R^E_+$ of \eqref{eq:jain} have $x(e) \geq 1/2$ for at least one $e \in
E$. 
Jain also noted that, in an implementation of his algorithm, the
computation of extreme points may be circumvented. In fact, he
suggests obtaining LP \lpg\ from \eqref{eq:jain} by adding the
constraint $x(g) \geq 1/2$ for some edge $g \in E$. 
Let $\opt_g$ be the objective function value of an optimum solution to
\lpg. Jain's structural lemma now implies that $\min_{g \in E}\opt_g$
is at most the optimum value of \eqref{eq:jain}. 
Jain's algorithm can now be implemented by replacing the computation
of an optimum basic solution to the residual problem in each
iteration, by computing optimal solutions to linear
programs of type \lpg\ for all edges $g \in E$. 

Of course, we can also replace computing an optimal solution to \lpg\ 
with computing an approximate one, at the expense of a slight increase  
of the final approximation factor. Jain's method in this case 
is summarized for completeness in \autoref{alg:jain}.

\begin{algorithm}
	\caption{A $2(1+\zeta)^{|E|}$- approximation algorithm for 
\eqref{eq:nsdp_ip}}
	\label{alg:jain}
	\begin{algorithmic}[1]
	\State $I_0\gets\varnothing$, $z_0(e)\gets 0$ for all  $e\in E$, 
$k\gets 0$
	\While{$f(S)-z_k(\delta(S))>0$ for some $S\subseteq 
V$}\label{state:app_while}
		\State $k\gets k+1$
		\ForAll{$g\in E\setminus I_{k-1}$} 
			\State \label{state:app_sol} find a
                        $(1+\zeta)$-approximation $x_{k,g} $ for \lpg\
                        with $z:= z_{k-1}$ and $I:= I_{k-1}$

		\EndFor
		\State let $x_k$ be a vector $x_{k,g}$ corresponding to 
$\min_{g\in E\setminus I_{k-1}} 
\sum_{e\in E} c(e) x_{k,g}(e)$
		\State $I_k\gets I_{k-1}\cup\{e\in E\setminus
                I_{k-1}\, | x_k(e)=0\,\text{ or }\, x_k(e)\ge 1/2 \}$
		\State \label{state:round} 
                for all $e\in E$, let $z_{k}(e)\gets\lceil x_k(E) \rceil$
                if $x_k(e) \geq 1/2$, and $z_k(e)\gets z_{k-1}(e)$ otherwise
	\EndWhile\\
	\Return $z_k$
	\end{algorithmic}
\end{algorithm}

\begin{claim}\label{cl:main}
Given $\zeta>0$, \autoref{alg:jain} is a 
$2(1+\zeta)^{|E|}$-approximation algorithm for \eqref{eq:nsdp_ip}. Moreover, 
\autoref{alg:jain} terminates after at most $|E|$ iterations of
step~\ref{state:app_while}.
\end{claim}
\begin{proof}
  Jain's structural lemma immediately implies that
  $I_{k-1}\subsetneq I_{k}\subseteq E$ for every $k$. 
  Hence the total
  number of iterations is at most $|E|$. In iteration $k$ we fix the
  values of variable $x(e)$, $e\in I_{k}\setminus I_{k-1}$, and due to the 
definition of $I_{k}\setminus I_{k-1}$ we
  have $z_{k}(e)\le 2 x_{k}(e)$, $e\in I_{k}\setminus I_{k-1}$. The
  remaining values $x_k(e)$, $e\in E\setminus I_k$ form a valid
  solution for~\eqref{eq:jain} with $z(e):=z_k(e)$, $e\in E$, which is
  solved with approximation guarantee $(1+\zeta)$ in the $(k+1)$-st
  iteration. Since there are at most $|E|$ iterations and in
  step~\ref{state:app_sol} the found solution is a
  \mbox{$(1+\zeta)$-approximation} of \lpg, we know that
  the objective value of the output is at most $2(1+\zeta)^{|E|}$
  times the objective value of the linear relaxation
  of~\eqref{eq:nsdp_ip}, finishing the proof. 
\end{proof}

\noindent
Note that by \autoref{cl:main}, if $\zeta\le\ln(1+\varepsilon)/|E|$ then 
\autoref{alg:jain} gives a $2(1+\varepsilon)$-approximation for 
\eqref{eq:nsdp_ip}. 

\section{Multiplicative weights method}\label{sec:mult_weight}

In this section, we briefly review the multiplicative weights 
method~\cite{Garg} of \gk, when applied to a positive
covering LP of the form
\begin{align}
	\min ~&\sum_{j\in[n]} c(j) x(j) 
\label{eq:cover_problem}\tag{$\mbox{LP}_2$}\\
	\st~&\sum_{j\in [n]}A(i,j) x(j) \ge b(i)\quad \forall i\in[m],
	\notag\\ 
	&x\ge 0\,, \notag
\end{align}
where $A(i,j)\ge 0$ for all $i\in[m]$, $j\in [n]$, $b(i)>0$ for all $i\in[m]$ and $c(j)>0$ for all 
$j\in[n]$. Note, that the linear program \lpg\ is a positive LP of the
above form when simply eliminating variables $x(e)$ for $e \in I$. 
In the same way, we can exclude the inequalities corresponding to 
$S\subseteq V$ with $f(S)-z(\delta(S))\le 0$.

Given $i\in[m]$ and a vector $x\in \R_{+}^n$ we 
define the \emph{length} $\len(i,x)$ of row $i$ with respect to~$x$ as
\begin{equation}\label{eq:len}
	\len(i,x):=\sum_{j\in[n]} A(i,j)x(j)/b(i)\,,
\end{equation}
and we denote by 
$\len(x)$ the shortest length of a row in $A$ with respect to $x$, i.e.\ 
$\len(x):=\min_{i\in [m]} \len(i,x)$.
Now it is straightforward to reformulate~\eqref{eq:cover_problem} as
\begin{equation}\label{eq:cover_problem_ref}
	\min_{x\ge 0, x\neq 0}\sum_{j\in[n]} c(j) x(j)/\len(x)\,.
\end{equation}

The multiplicative weights method of \gk\ applied to the dual
of~\eqref{eq:cover_problem} computes an approximate pair of primal and
dual solutions in strongly-polynomial time, as long as it is provided
with a strongly-polynomial time oracle for determining the row $q$ of
shortest length (the \emph{shortest row}) with respect to given
lengths $x\in \R^n_+$ as in \eqref{eq:len}. 

It is implicit in the work of \cite{Garg, Fleischer} that exact
oracles can be replaced by approximate ones (incurring a corresponding
degradation in performance ratio, of course).  Such a modification is
described from a packing point of view in \cite{EMR15}, for example. 
\autoref{alg:mult_weight} shows the pseudo code of the algorithm
for completeness. In step~\ref{state:approx} of the algorithm a 
\mbox{$(1+\zeta)$-approximation} $q$ of the shortest row with respect to 
some vector $x\in\R_{+}^n$ is computed. That is, $q$ is a row for which 
$\len(q,x)\leq(1+\zeta)\len(x)$. \autoref{sec:shortest_row} describes how 
to obtain this approximation in strongly-polynomial time. 
\begin{algorithm}
	\caption{The multiplicative weights algorithm to 
solve~\eqref{eq:cover_problem}.}
		\label{alg:mult_weight}
	\begin{algorithmic}[1]
	\State $\delta\gets(1+\zeta)\big((1+\zeta)n\big)^{-\frac{1}{\zeta}}$, 
$x_0(j)\gets \delta/c(j)$ for all $j\in[n]$, $y_0(i)\gets 0$ for all $i\in[m]$, 
$k\gets 0$
	\While{$\sum_{j\in[n]} c(j) x_k(j)<1$}\label{state:mult_app}
		\State $k\gets k+1$
		\State determine a $(1+\zeta)$-approximation for the shortest 
row with respect to $x_{k-1}$, let it be row $q_k$ \label{state:approx}
		\State determine $j\in[n]$ with the minimum value 
$c(j)/A(q_k,j)$, let it be column $p_k$
		\ForAll{$i\in [m]$}
$$y_{k}(i)\gets
						\begin{cases} 
							
y_{k-1}(i)+c(p_k)/A(q_k,p_k) & \text{if}\quad i=q_k\\
							
y_{k-1}(i)&\text{otherwise}\,.
						\end{cases}$$ 
		\EndFor
		\State $x_{k}(j)\gets  \big( 1+\zeta 
\frac{c(p_k)A(q_k,j)}{c(j)A(q_k,p_k)} \big)x_{k-1}(j)$ for all $j\in [n]$
	\EndWhile\\
	\Return $x_k/\len(x_k)$ corresponding to $\min_{k}\sum_{j\in[n]} c(j) 
x_k(j)/\len(x_k)$
	\end{algorithmic}
\end{algorithm}
We give a proof of the next lemma for completeness. 

\begin{lemma}[implicit in \cite{Garg,Fleischer}] 
\label{cl:mult_weight}
\autoref{alg:mult_weight} is a $(1+4\zeta)$-approximation 
for~\eqref{eq:cover_problem}. Moreover, \autoref{alg:mult_weight} 
terminates after at most~$\frac{1}{\zeta}\log_{1+\zeta}(1+\zeta)n$ iterations.
\end{lemma}
\begin{proof}
Let us define $\beta:=\min_{k}\frac{\sum_{j\in[n]} c(j) x_k(j)}{\len(x_k)}$, 
below we show that $\beta$ provides a good approximation for the 
problem given by~\eqref{eq:cover_problem_ref}. 

For every $k\ge 1$ we have
\begin{multline*}
	\sum_{j\in[n]} c(j) x_k(j)- \sum_{j\in[n]} c(j) x_{k-1}(j)=\\
	\zeta \len(q_k,x_{k-1})c(p_k) b(q_k)/A(q_k,p_k)\le
	\zeta (1+\zeta) \len(x_{k-1})\times \sum_{i\in[m]}b(i)(y_{k}(i)-y_{k-1}(i)) 
\,.
\end{multline*}

Hence,
\[
	\sum_{j\in[n]} c(j) x_k(j) \le \sum_{j\in[n]} c(j) x_0(j)+ \zeta 
(1+\zeta)\times \sum_{h\in [k]} \sum_{i\in[m]}b(i)(y_{h}(i)-y_{h-1}(i)) 
\len(x_{h-1})\,.
\]
Due to the definition of $\beta$ we have $\len(x_{h-1})\le \sum_{j\in[n]} c(j) 
x_{h-1}(j)/\beta$ and thus
\[\label{eq:induct_step}
\sum_{j\in[n]} c(j) x_k(j) \le\\ \sum_{j\in[n]} c(j) x_0(j)+\frac{ \zeta 
(1+\zeta)}{\beta}\times \sum_{h\in [k]} 
\sum_{i\in[m]}b(i)(y_{h}(i)-y_{h-1}(i))\sum_{j\in[n]} c(j) x_{h-1}(j)\,.
\]

To show that the right-hand side  of \eqref{eq:induct_step} is at 
most $n\delta \me^{\zeta (1+\zeta) \sum_{i\in[m]}b(i)y_{k}(i)/\beta}$, we use 
induction. Indeed, 
the case $k=0$ is clear, and to show the statement consider
\[
	\sum_{j\in[n]} c(j) x_0(j)+ \frac{\zeta(1+\zeta)}{\beta}\times
	\sum_{h\in [k]}\sum_{i\in[m]}b(i)(y_{h}(i)-y_{h-1}(i))\sum_{j\in[n]} 
c(j) x_{h-1}(j)\,,
\]
which equals
\begin{multline*}
	\sum_{j\in[n]} c(j) x_0(j)+ \frac{\zeta(1+\zeta)}{\beta}\times 
	\Big(\sum_{h\in [k-1]}\sum_{i\in[m]}b(i)(y_{h}(i)-y_{h-1}(i))
	\sum_{j\in[n]} c(j) x_{h-1}(j)+ \\
\sum_{i\in[m]}b(i)\big(y_{k}(i)-y_{k-1}(i)\big) \sum_{j\in[n]} c(j) 
x_{k-1}(j)\Big)
\end{multline*}
Due to~\eqref{eq:induct_step} we conclude that the last expression is at most
\begin{multline*}
\Big(1+\frac{\zeta(1+\zeta)}{\beta}\sum_{i\in[m]}b(i)\big(y_{k}(i)-y_{k-1}
(i)\big)\Big)\times
	\Big(\sum_{j\in[n]} c(j) x_0(j)+ \frac{\zeta(1+\zeta)}{\beta}\times\\
	\sum_{h\in [k-1]} 
\sum_{i\in[m]}b(i)\big(y_{h}(i)-y_{h-1}(i)\big)\sum_{j\in[n]} c(j) 
x_{h-1}(j)\Big)\,.
\end{multline*}
Using the inequality $(1+\alpha)\le \me^\alpha$, $\alpha\in\R$ and the 
induction hypothesis we upper-bound the expression above by
\[
	\me^{\zeta(1+\zeta) 
          \sum_{i\in[m]}b(i)\big(y_{k}(i)-y_{k-1}(i)\big)/\beta}\times
        n\delta 
        \me^{\zeta(1+\zeta)  \sum_{i\in[m]}b(i)y_{k-1}(i)/\beta}=
	n\delta \me^{\zeta(1+\zeta)  \sum_{i\in[m]}b(i)y_{k}(i)/\beta}\,.
\]

Now let us consider the last iteration $t$, where we have
$$
 1 \le \sum_{j\in[n]} c(j) x_t(j) \le n\delta \me^{\zeta (1+\zeta) 
\sum_{i\in[m]}b(i)y_{t}(i)/\beta}\,,
$$
and thus
\begin{equation}\label{dual primal ratio}
	\frac{\beta}{\sum_{i\in[m]}b(i)y_{t}(i)}\le 
\frac{\zeta(1+\zeta)}{\ln((n\delta)^{-1})}
\end{equation}
whenever $n\delta<1$.

Now let us show that $y_t/\log_{1+\zeta}\big(\frac{1+\zeta}{\delta} \big)$ is a 
feasible solution for the dual of~\eqref{eq:cover_problem}. It is enough to 
show that 
$$
	\max_{j}\sum_{i\in[m]} \frac{A(i,j) y_t(i)}{c(j)}\le 
\log_{1+\zeta}\frac{1+\zeta}{\delta}\,.
$$
To see this note that for every $j\in[n]$ and every $k$
\[
	\sum_{i\in[m]} \frac{A(i,j) y_k(i)}{c(j)}-\sum_{i\in[m]} \frac{A(i,j) 
y_{k-1}(i)}{c(j)}=\frac{A(q_k,j)}{c(j)}\frac{c(p_k)}{A(q_k,p_k)}\leq 1\,,
\]
and $\sum_{i\in[m]} \frac{A(i,j) y_{0}(i)}{c(j)}=0$. On the other hand for 
every $j\in[n]$ and every $k$
$$
\frac{x_k(j)}{x_{k-1}(j)}=1+\zeta \frac{c(p_k)A(q_k,j)}{c(j)A(q_k,p_k)}\leq 
1+\zeta\,,
$$
$x_0(j)=\delta/c(j)$ and due to the termination condition $x_{t-1}(j)<1/c(j)$ 
and hence $x_{t}(j)<(1+\zeta)/c(j)$. This implies that the algorithm terminates 
after at most $\log_{1+\zeta}\frac{1+\zeta}{\delta}$ iterations. Thus, $y_t/\log_{1+\zeta}\big(\frac{1+\zeta}{\delta} \big)$ is a feasible solution for the dual of~\eqref{eq:cover_problem}.

Hence, the algorithm provides a feasible solution 
for~\eqref{eq:cover_problem_ref} with value $\beta$, which is an approximation 
with guarantee
$$
\frac{\zeta(1+\zeta)}{\ln((n\delta)^{-1})}\log_{1+\zeta}\frac{1+\zeta}{\delta}
=\frac{\zeta(1+\zeta)}{\ln (1+\zeta)}\frac{\ln 
\frac{1+\zeta}{\delta}}{\ln((n\delta)^{-1})}\,,
$$
due to~\eqref{dual primal ratio} and the fact that $y_t/\log_{1+\zeta}\big(\frac{1+\zeta}{\delta} \big)$ is a 
feasible solution for the dual of~\eqref{eq:cover_problem}.
Thus, we obtain
\[
\frac{\zeta(1+\zeta)}{\ln (1+\zeta)}\frac{\ln 
\frac{1+\zeta}{\delta}}{\ln((n\delta)^{-1})}= 
\frac{\zeta(1+\zeta)}{(1-\zeta)\ln (1+\zeta)}\le
  \frac{\zeta(1+\zeta)}{(1-\zeta)(\zeta-\zeta^2/2)}\le 
\frac{1+\zeta}{(1-\zeta)^2}\,,
\]
which is at most $(1+4\zeta)$ for $\zeta\le 0.15$.
\end{proof}

\section{The shortest row problem}\label{sec:shortest_row}

In this section we describe how to (approximately) solve the shortest row 
problem needed in 
\autoref{alg:mult_weight} when applied to \eqref{eq:jain}. We start by 
stating the following simple remark,
that we will need at the end of our analysis.

\begin{rem}\label{rem:function_values}
For  every $k\ge 1$ and every $S\subseteq V$, we have
$f(S)-z_k(\delta(S))\le |E|/2$ in \autoref{alg:jain}.
\end{rem}
\begin{proof}
Define $x \in \R^E_+$ by letting $x(e):=x_k(e)$ whenever $x_k(e)<1/2$,
and let $x(e):=0$ otherwise. Then note that 
$  |E|/2 \ge x(\delta(S)) \ge f(S)-z_k(\delta(S))$,
due to the feasibility of $x$ in~\eqref{eq:jain} with $z:=z_k$ and $I:=I_k$. 
\end{proof}

Let us recall that, for given $x \in \R^E_+$, $z \in \Z^E_+$, and proper 
function $f$, the shortest row problem we need to solve is the following:
\[
	\min_{\substack{f(S)-z(\delta(S))\ge 1\\S\subseteq V}} 
\frac{x(\delta(S))}{f(S)-z(\delta(S))}.
\]

\noindent
The above shortest row problem is quite easy to solve when $z=\0$. 
In this case, Gabow et al.~\cite{Gabow} 
give a strongly-polynomial time separation oracle based on 
the construction of \emph{Gomory-Hu trees}~\cite{GomoryHu}, as we are now going 
to explain.

Given a graph $G=(V,E)$ and values $x(e)\in\R_+$ for each $e\in E$, a 
\emph{Gomory-Hu tree}~\cite{GomoryHu} is a capacitated tree $T=(V,J)$ such that 
for any two 
vertices $v,u\in V$ the minimum $x$-value of a cut in $G$ 
separating $v$ and $u$ equals the minimum $x$-value among the $u$-$v$ cuts 
induced by the edges of $T$. More 
concretely, let 
$S_e$ and $V\setminus S_e$ induce connected components in $T$ after removing 
$e$ from $T$. We then have
$$
	\min_{\substack{u\in S, v\not\in S\\S\subseteq V}} x(\delta(S))= 
\min_{\substack{u\in S_e, v\not\in S_e\\e\in J}} x(\delta(S_e))\,.
$$


The next lemma shows that in order to find the shortest row in the first 
iteration of step~\ref{state:app_while} in \autoref{alg:jain} (i.e. when 
$z=\0$), it is 
enough to compute a Gomory-Hu tree with respect to values $x(e)\in\R_+$, 
$e\in E$.
\begin{lemma}[\cite{Gabow}]\label{lem:gabow}
Given a graph $G=(V,E)$, a proper function $f:2^V\rightarrow \Z_+$ and a 
Gomory-Hu tree $T=(V,J)$ with respect to values $x(e)\in\R_+$, $e\in E$, we have
$$
	\min_{\substack{f(S)\neq 0\\S\subseteq V}} 
x(\delta(S))/f(S)=\min_{\substack{f(S_e)\neq 0\\e\in J}} 
x(\delta(S_e))/f(S_e)\,.
$$
\end{lemma}
\begin{proof}
Consider $S\subseteq V$ and the edges $\delta_T(S)$ in the Gomory-Hu tree $T$ 
defined by $S$. By definition of a Gomory-Hu tree $x(\delta(S))\ge 
x(\delta(S_e))$ for every \mbox{$e\in \delta_T(S)$}, due to the cut in $T$ 
incurred by the vertices incident to $e$. Thus, to prove the claim it is enough 
to show that 
\begin{equation}\label{eq:gomory_hu_proper}
	f(S)\le \max_{e\in \delta_T(S)} f(S_e)\,.
\end{equation}
To show the last inequality let $V_1$,\ldots,$V_k$ be vertex sets of the 
connected components after  removing $S$ in $T$. Thus, $V_1$,\ldots, $V_k$ form 
a partition of $V\setminus S$, and so
$\max_{i\in [k]} f(V_i) \ge f(V\setminus S)=f(S)$.
Choose $i\in [k]$. Replacing $S_e$ by $V\setminus S_e$, we can assume that 
$S_e$ and $V_i$ are disjoint for every $e\in\delta_T(V_i)$. Thus the sets 
$S_e$ with $e\in \delta_T(V_i)$ partition $V\setminus V_i$, showing that
$f(V_i)\le  \max_{e\in\delta_T(V_i)} f(S_e).$
Since, $\delta_T (V_i)$, $i\in [k]$ partition $\delta(S)$ we 
get~\eqref{eq:gomory_hu_proper}, finishing the proof.
\end{proof}

In the later iterations of steps~\ref{state:app_while} in
\autoref{alg:jain}, the inequality corresponding to $S\subseteq V$ has 
the form $x(\delta(S))\ge g(S)$, where $g:2^V\rightarrow \Z_+$  is such that
$g(S) = f(S)-z(\delta(S))$ for some $z(e)\in\Z_+$, $e\in E$, and a proper 
function $f$. Once $z\neq\0$, $g(S)$ is not a proper function any more and 
unfortunately, Gabow et al.'s algorithm can not be used directly. We do not 
know how to solve this problem \emph{exactly} in strongly-polynomial time,
 but we can approximate it using the following observation. 

Fix a value $\gamma >0$, and let us check
whether the optimal solution of the shortest row problem has 
a value less than $\gamma$. The crucial fact is that given $x(e)\in\R_+$, $e\in 
E$ and $\gamma>0$ checking 
whether
$$
	\min_{\substack{f(S)-z(\delta(S))\ge 1\\S\subseteq V}} 
\frac{x(\delta(S))}{f(S)-z(\delta(S))} < \gamma\,.
$$
is equivalent to checking whether 
$$
	x(\delta(S))/\gamma+ z(\delta(S)) < f(S)
$$
for some $S\subseteq V$, i.e.\ it can be reduced to finding
$$
	\min_{\substack{f(S)\neq 0\\S\subseteq V}} 
\frac{x(\delta(S))/\gamma + z(\delta(S))}{f(S)}.
$$
Therefore, we can apply \autoref{lem:gabow} after replacing $x(e)$ with 
$x(e)/\gamma+z(e)$. 
This enables us to use binary search to find a $(1+\zeta)$-approximation for 
the shortest row indexed 
by vertex subsets whenever we have a lower bound $\gamma_{\min}$ and an upper 
bound $\gamma_{\max}$ on the length of the shortest row. Giving trivial bounds 
on such a value (e.g. $1$ and $(|E| \cdot \max_{S} f(S))$) 
is of course easy. However, given an interval $[\gamma_{\min}, \gamma_{\max}]$ 
for binary search we have to construct a Gomory-Hu tree $\lceil\log_{1+\zeta} 
\gamma_{\max}/\gamma_{\min}\rceil$ times, and therefore we need
that $\gamma_{\max}/\gamma_{\min}$ is independent of the size
of $f$ in order to achieve strong polynomiality. To this aim, we propose 
\autoref{alg:upper_lower_bounds}. 

\begin{algorithm}
	\caption{Determining upper and lower bounds $\gamma_{\min}$, 
$\gamma_{\max}$.}
	\label{alg:upper_lower_bounds}
	\begin{algorithmic}[1]
	\State  $G_0=(V_0,E_0) \gets G=(V,E)$, $f_0(S)\gets f(S)$ for all 
$S\subseteq V$, $k\gets 0$
	\While{$f_{k}(S)-z(\delta_{G_k}(S))> 0$ for some $S\subseteq V_k$}
		\State find a set $S\subseteq V_{k}$ such that 
$f_{k}(S)-z(\delta_{G_{k}}(S))\ge 1$, let it be $S_k$
		\State determine $e\in \delta_{G_{k}}(S_k)$ with maximum 
$x(e)$, let it be $e_k\gets\{u_k,v_k\}$
		\State contract $e_k$ in $G_{k}$ (keeping multiple copies of 
edges) to obtain $G_{k+1}$, and set
		$$
			f_{k+1}(S)\gets\begin{cases}
					f_{k}(S\cup \{u_k,v_k\}\setminus 
w_k)&\text{if}\quad w_k\in S\\
					f_{k}(S)&\text{otherwise}\,,
				\end{cases}
		$$
	for all $S\subseteq V_{k+1}$, where $w_k$ is the vertex in $G_{k+1}$ 
corresponding to the contracted edge~$e_k$.
	\State $k\gets k+1$
	\EndWhile
	\State $\gamma_{\min}\gets \min_{k}2 x(e_k)/|E|$, and let $p$ be the 
index for which this minimum is achieved
	\State $\gamma_{\max}\gets x(\delta(U_p))$, where $U_p$ is the vertex 
subset of $V$ corresponding to the vertex subset $S_p$ of $V_p$\\
	\Return $\gamma_{\min}$, $\gamma_{\max}$
	\end{algorithmic}
\end{algorithm}

\begin{lemma}
\autoref{alg:upper_lower_bounds} computes an interval $[\gamma_{\min}, 
\gamma_{\max}]$, which contains the shortest row length with respect to 
$x(e)\in \R_+$, 
$e\in E$. Moreover, $\gamma_{\max}/\gamma_{\min}\le |E|^2/2$, and the algorithm 
runs in strongly-polynomial time.
\end{lemma}

\begin{proof}
\autoref{alg:upper_lower_bounds} works as follows. It does a sequence of 
at most $|V|$
iterations. In iteration $k$, it takes an arbitrary subset $S$ corresponding to 
a
violated cut, i.e. such that $f_{k}(S)-z(\delta_{G_k}(S)) >0$, and contracts 
the edge $e_k$ in this cut of maximum $x$-value. Contracting
this edge naturally yields a graph $G_{k+1}$ and a function $f_{k+1}$ to use in 
the next iteration. 
Note that a violated subset $S$ can be computed efficiently given that $f$ is a 
proper function~\cite{GW95}.

Our first claim is that $\gamma_{\min}$ is a valid lower bound on the shortest 
row length.
In other words, we claim that for every $S: f(S)-z(\delta(S))\ge 1$, we have
$$\frac{x(\delta(S))}{f(S)-z(\delta(S))}\ge 
\frac{x(e_p)}{|E|/2}=\gamma_{\min}\,.
$$

Due to the termination condition, for every $S\subseteq V$ with 
$f(S)-z(\delta(S))\ge 1$ the edge set $\delta(S)$ contains at least one of the 
edges $e_1$,\ldots,$e_t$ selected by the algorithms during its $t$ iterations. 
Therefore, $x(\delta(S)) \geq x(e_p)$, by the choice of $p$ in step 8.
Moreover, by \autoref{rem:function_values}  $f(S)-z(\delta(S))\le |E|/2$. 
The claim then follows. 

Our second claim is that $\gamma_{\max}$ is a valid upper bound on the shortest 
row length.
To see this, note that $f(U_p)-z(\delta(U_p))\ge 1$ 
because $f(U_p)=f_{p}(S_p)$ and $z(\delta(U_p))=z(\delta_{G_{p}}(S_p))$, 
proving that 
$\gamma_{\max}$ is a valid upper bound for the shortest length of a row indexed 
by $S\subseteq V$.

Finally, recalling that $e_p$ satisfies $x(e_p) = \max_{e\in\delta(U_p)}x(e)$ 
(step 4), we have 
$$
	\gamma_{\max}/\gamma_{\min}=\frac{x(\delta(U_p))}{2 x(e_p)/|E|}\le 
|E|^2/2\,. \qedhere
$$
\end{proof}

\section{Concluding remarks}\label{sec:wrap-up}

We are now ready to put all pieces together and give a proof of 
\autoref{thm:main} and
\autoref{crl:main} stated in the
introduction.

\begin{proof}[Proof of \autoref{thm:main}.]
Given an $\varepsilon >0$, we apply \autoref{alg:mult_weight} to 
\eqref{eq:jain} with 
$\zeta=\ln(1+\varepsilon)/|E|$. 
\autoref{alg:mult_weight} in its turn approximates the
shortest row at most $O((\ln |V|)/\zeta^2)$ times\footnote{Here, we
  use the inequality that $1+\zeta<\me^\zeta<1+\frac{7}{4}\zeta$ for
  $0<\zeta<1$.}. It makes a call to
\autoref{alg:upper_lower_bounds}, computing at most $|V|$
Gomory-Hu trees and afterwards the binary search needs
$O((\ln|E|)/\zeta)$ computations of a Gomory-Hu tree in
$G=(V,E)$. Recall, that
$\zeta=\ln(1+\varepsilon)/|E|=\Theta(\varepsilon/|E|)$ and hence each
linear program appearing in \autoref{alg:jain} is solved in time
dominated by finding $O(|E|^3(\ln |E|)^2/\varepsilon^3)$ Gomory-Hu
trees. 
Note that a Gomory-Hu tree for $G=(V,E)$ with 
respect to values $x(e)\in\R_+$, $e\in E$ can be found by $|V|$ computations of 
the minimum cut in $G$~\cite{GomoryHu}, so a Gomory-Hu tree can be found in 
strongly-polynomial time. %
\end{proof}

The number of times our algorithm solves the Gomory-Hu tree
problem is substantially smaller than the corresponding number for the
Ellipsoid method given the classical estimation for the encoding
length of vertices or given that $\max_{S} f(S)$ is sufficiently large, because
this number for the Ellipsoid method grows proportionally with the logarithm of 
$\max_{S} f(S)$.



\begin{proof}[Proof of \autoref{crl:main}.]
To obtain a \mbox{$(2+\varepsilon)$-approximation}
for~\eqref{eq:nsdp_ip} we apply \autoref{alg:jain}. 
\autoref{alg:jain} solves $O(|E|^2)$ linear programs, i.e.
there are $O(|E|^2)$ calls of \autoref{alg:jain} to
\autoref{alg:mult_weight} to solve linear
programs, and it makes at most $|E|$ roundings.
Considering rounding as a basic operation, the result follows.
\end{proof}


\medskip
We conclude the paper with some open questions.
It remains open whether one is able to provide a
$2$-approximation algorithm for \eqref{eq:nsdp_ip}, which does not need to solve 
linear programs. This question is among the top 10 
open questions in the theory of approximation algorithms according to Shmoys 
and 
Williamson~\cite{Williamson}. In our opinion, a good intermediate question is 
whether it is possible 
to give an algorithm with a constant approximation guarantee such 
that  the number of linear programs solved in its course is bounded by a 
constant. One way to prove this could be to exploit that after each rounding 
in 
the 
algorithm of Jain~\cite{Jain} we have a sufficiently ``good'' feasible point 
for 
the new linear program.
\bibliography{sndp}

\begin{thebibliography}{10}

\bibitem{AHK12}
S.~Arora, E.~Hazan, and S.~Kale.
\newblock The multiplicative weights update method: a meta-algorithm and
  applications.
\newblock {\em Theory of Computing}, 8(1):121--164, 2012.

\bibitem{Bi06}
D.~Bienstock.
\newblock {\em Potential function methods for approximately solving linear
  programming problems: theory and practice}, volume~53.
\newblock Springer Science \& Business Media, 2006.

\bibitem{BI04}
D.~Bienstock and G.~Iyengar.
\newblock Solving fractional packing problems in o ast (1/$\varepsilon$)
  iterations.
\newblock In {\em Proceedings of the thirty-sixth annual ACM symposium on
  Theory of computing}, pages 146--155. ACM, 2004.

\bibitem{EMR15}
K.~Elbassioni, K.~Mehlhorn, and F.~Ramezani.
\newblock Towards more practical linear programming-based techniques for
  algorithmic mechanism design.
\newblock {\em SAGT}, pages 98--109, 2015.

\bibitem{Fl04}
L.~Fleischer.
\newblock A fast approximation scheme for fractional covering problems with
  variable upper bounds.
\newblock In {\em Proceedings of the fifteenth annual ACM-SIAM symposium on
  Discrete algorithms}, pages 1001--1010. Society for Industrial and Applied
  Mathematics, 2004.

\bibitem{Fleischer}
L.~K. Fleischer.
\newblock Approximating fractional multicommodity flow independent of the
  number of commodities.
\newblock {\em SIAM Journal on Discrete Mathematics}, 13(4):505--520, 2000.

\bibitem{Gabow}
H.~N. Gabow, M.~X. Goemans, and D.~P. Williamson.
\newblock An efficient approximation algorithm for the survivable network
  design problem.
\newblock {\em Mathematical Programming}, 82(1-2, Ser. B):13--40, 1998.
\newblock Networks and matroids; Sequencing and scheduling.

\bibitem{GK02}
N.~Garg and R.~Khandekar.
\newblock Fast approximation algorithms for fractional steiner forest and
  related problems.
\newblock In {\em Foundations of Computer Science, 2002. Proceedings. The 43rd
  Annual IEEE Symposium on}, pages 500--509. IEEE, 2002.

\bibitem{GK04}
N.~Garg and R.~Khandekar.
\newblock Fractional covering with upper bounds on the variables: Solving lps
  with negative entries.
\newblock In {\em Algorithms--ESA 2004}, pages 371--382. Springer, 2004.

\bibitem{Garg}
N.~Garg and J.~K{\"o}nemann.
\newblock Faster and simpler algorithms for multicommodity flow and other
  fractional packing problems.
\newblock {\em SIAM J. Comput.}, 37(2):630--652, 2007.

\bibitem{GSGT1994}
M.~X. Goemans, D.~B. Shmoys, A.~V. Goldberg, {\'E}.~Tardos, S.~Plotkin, and
  D.~P. Williamson.
\newblock Improved approximation algorithms for network design problems.
\newblock In {\em Proceedings of the {F}ifth {A}nnual {ACM}-{SIAM} {S}ymposium
  on {D}iscrete {A}lgorithms ({A}rlington, {VA}, 1994)}, pages 223--232. ACM,
  New York, 1994.

\bibitem{GW95}
M.~X. Goemans and D.~P. Williamson.
\newblock A general approximation technique for constrained forest problems.
\newblock {\em SIAM Journal on Computing}, 24(2):296--317, 1995.

\bibitem{GomoryHu}
R.~E. Gomory and T.~C. Hu.
\newblock Multi-terminal network flows.
\newblock {\em J. Soc. Indust. Appl. Math.}, 9:551--570, 1961.

\bibitem{GK94}
M.~D. Grigoriadis and L.~G. Khachiyan.
\newblock Fast approximation schemes for convex programs with many blocks and
  coupling constraints.
\newblock {\em SIAM Journal on Optimization}, 4(1):86--107, 1994.

\bibitem{GK96}
M.~D. Grigoriadis and L.~G. Khachiyan.
\newblock Approximate minimum-cost multicommodity flows in $\tilde o(\epsilon -
  2 knm)$ time.
\newblock {\em Mathematical Programming}, 75(3):477--482, 1996.

\bibitem{GLS88}
M.~Gr{\"o}tschel, L.~Lov{\'a}sz, and A.~Schrijver.
\newblock {\em Geometric algorithms and combinatorial optimization}, volume~2
  of {\em Algorithms and Combinatorics: Study and Research Texts}.
\newblock Springer-Verlag, Berlin, 1988.

\bibitem{Gu11}
A.~Gupta and J.~K{\"o}nemann.
\newblock Approximation algorithms for network design: A survey.
\newblock {\em Surveys in Operations Research and Management Science},
  16(1):3--20, 2011.

\bibitem{Ho96}
D.~S. Hochbaum.
\newblock {\em Approximation algorithms for {NP}-hard problems}.
\newblock PWS Publishing Co., 1996.

\bibitem{Jain}
K.~Jain.
\newblock A factor 2 approximation algorithm for the generalized {S}teiner
  network problem.
\newblock {\em Combinatorica}, 21(1):39--60, 2001.

\bibitem{LN93}
M.~Luby and N.~Nisan.
\newblock A parallel approximation algorithm for positive linear programming.
\newblock In {\em Proceedings of the twenty-fifth annual ACM symposium on
  Theory of computing}, pages 448--457. ACM, 1993.

\bibitem{SM90}
F.~Shahrokhi and D.~W. Matula.
\newblock The maximum concurrent flow problem.
\newblock {\em Journal of the ACM (JACM)}, 37(2):318--334, 1990.

\bibitem{Ta86}
E.~Tardos.
\newblock A strongly polynomial algorithm to solve combinatorial linear
  programs.
\newblock {\em Operations Research}, 34(2):250--256, 1986.

\bibitem{Williamson}
D.~P. Williamson and D.~B. Shmoys.
\newblock {\em The design of approximation algorithms}.
\newblock Cambridge University Press, Cambridge, 2011.

\bibitem{Yo95}
N.~E. Young.
\newblock Randomized rounding without solving the linear program.
\newblock In {\em SODA}, volume~95, pages 170--178, 1995.

\bibitem{Yo01}
N.~E. Young.
\newblock Sequential and parallel algorithms for mixed packing and covering.
\newblock In {\em Foundations of Computer Science, 2001. Proceedings. 42nd IEEE
  Symposium on}, pages 538--546. IEEE, 2001.

\end{thebibliography}
\bibliographystyle{abbrv}
\end{document}